\documentclass[12pt,reqno]{amsart}
\usepackage{amsmath,amssymb,extarrows}
\usepackage{hyperref}
\usepackage{url}
\usepackage{tikz,enumerate}
\usepackage{diagbox}
\usepackage{appendix}
\usepackage{epic}
\usepackage{comment}
\usepackage{bm}
\usepackage{adjustbox}
\usepackage{booktabs}
\usepackage{array}
\usepackage{mathrsfs} 

\usepackage{float} 

\usepackage{cite}
\usepackage{hyperref}
\usepackage{array}

\allowdisplaybreaks[4]

\newtheorem{theorem}{Theorem}
\newtheorem{corollary}[theorem]{Corollary}

\newtheorem{lemma}[theorem]{Lemma}

\theoremstyle{definition}

\theoremstyle{remark}

\numberwithin{equation}{section}
\numberwithin{theorem}{section}
\numberwithin{defn}{section}

\begin{document}
\title[Proofs of four Takigiku--Tsuchioka conjectures]
 {Proofs of four Takigiku--Tsuchioka conjectures on level $2$ principal characters of $A_{2\ell-1}^{(2)}$}

\author{Liuquan Wang}
\address{School of Mathematics and Statistics, Wuhan University, Wuhan 430072, Hubei, People's Republic of China}
\email{wanglq@whu.edu.cn;mathlqwang@163.com}




\begin{abstract}
We prove four conjectures of Takigiku and Tsuchioka consisting of eight Rogers--Ramanujan type identities. Two conjectures involve triple sums and a third involves quadruple sums, and their product sides are certain level $2$ principal characters of $A_{13}^{(2)}$. The fourth conjecture gives multisum representations of certain level $2$ principal characters of $A_{2\ell-1}^{(2)}$. The proofs rely on the Bailey pair method, finite sum identities involving Gaussian polynomials, and the quintuple product identity.
\end{abstract}

\maketitle

\section{Introduction and main results}\label{sec-intro}

Rogers--Ramanujan type identities are an important topic in the theory of $q$-series,  with deep connections to number theory, combinatorics, the representation theory of Lie algebras, and mathematical physics. Their study dates back to the celebrated Rogers--Ramanujan identities:
\begin{align}\label{RR}
    \sum_{n=0}^\infty \frac{q^{n^2+\delta n}}{(q;q)_n}=\frac{1}{(q^{\delta+1},q^{4-\delta};q^5)_\infty}, \quad \delta=0,1.
\end{align}
Here and throughout, we assume $|q|<1$ and use the standard $q$-series notation:
\begin{align}(a_1,a_2,\dots,a_m;q)_n:=\prod\limits_{i=1}^m\prod\limits_{k=1}^n (1-a_iq^{k-1}), \quad n\in \mathbb{Z}_{\geq 0}\cup\{\infty\}.
\end{align}
These identities were first proved by Rogers \cite{Rogers1894} in 1894 and later rediscovered by Ramanujan and Schur; see Sills' book \cite{Sills-book} for more historical background. 

The Rogers--Ramanujan identities can also be interpreted as partition identities equating the numbers of partitions of any given integer into two kinds of partitions: one defined by difference conditions on adjacent parts and the other by congruence conditions modulo $5$. For example, the first identity is equivalent to saying that, for each nonnegative integer $N$, the number of partitions of $N$ in which successive parts differ by at least two equals the number of partitions of $N$ into parts congruent to $1$ or $4$ modulo $5$. Gordon \cite{Gordon1961} generalized \eqref{RR} to partition identities for any odd modulus greater than $3$. Its equivalent analytic form was proved by Andrews \cite{Andrews1974}. An even modulus analog was then found by Bressoud \cite{Bressoud1980}. The Andrews--Gordon--Bressoud identities \cite{Andrews1974,Gordon1961,Bressoud1980} give a unified $q$-series formulation of these results:  for $\delta=0,1$, $k\geq 2$ and $1\leq i \leq k$,
\begin{equation}\label{AGB}
\sum_{n_1,\dots,n_{k-1}\geq 0}
\frac{q^{N_1^2+\cdots+N_{k-1}^2+N_i+\cdots+N_{k-1}}}{(q;q)_{n_1}\cdots(q;q)_{n_{k-2}}
(q^{\,2-\delta};q^{\,2-\delta})_{n_{k-1}}}=
\frac{(q^i,q^{2k+\delta-i},q^{2k+\delta};q^{2k+\delta})_\infty}{(q;q)_\infty},
\end{equation}
where $N_j=n_j+n_{j+1}+\cdots+n_{k-1}$.
These identities motivate the search for similar multisum Rogers--Ramanujan type identities. 

The theory of Bailey pairs \cite{Bailey1949} are among the principal analytic tools in the study of Rogers--Ramanujan type identities. A Bailey pair relative to $a$ with base $q$ is a pair of sequences $(\alpha_n(a;q),\beta_n(a;q))$ satisfying for all $n\geq 0$ that
\begin{equation}\label{BP-defn}
 \beta_n(a;q)=\sum_{r=0}^{n}
 \frac{\alpha_r(a;q)}{(q;q)_{n-r}(aq;q)_{n+r}}.
\end{equation}
Transformations of Bailey pairs and their limiting cases produce infinite sum--product identities from finite summation relations. A famous application is Slater's work \cite{Slater1952} wherein Bailey pairs were used to establish 130 Rogers--Ramanujan type identities. The theory was subsequently extended to include Bailey chains \cite{Andrews1984} and Bailey lattices \cite{AAB}. We refer the reader to Warnaar's survey \cite{Warnaar-survey} for a detailed introduction on Bailey pairs.

Affine Lie algebras provide an algebraic approach to the study of Rogers--Ramanujan type identities. The work of Lepowsky and Milne~\cite{LM1978} and the vertex-operator constructions of Lepowsky and Wilson~\cite{LW1982,LW1984,LW1985} established connections between the Rogers--Ramanujan identities and principally specialized characters. Character formulas provide candidate infinite products and suitably chosen monomial bases give rise to partition restrictions and multiple series. A partition identity can then be interpreted as the equality of two expressions for the same graded character. This motivates the search for explicit series expansions and partition-theoretic interpretations of character products for different affine types and levels. In particular, the principal character $\chi_{X_N^{(r)}}(\lambda)$ of an integrable highest weight module with highest weight $\lambda$ over an affine Lie algebra of type $X_N^{(r)}$ is expected to admit a Rogers--Ramanujan type sum representation~\cite[Sec.~10]{LW1984}. For example, the product side of the Andrews--Gordon--Bressoud identities \eqref{AGB} equals $\chi_{A_1^{(1)}}((2k-1+\delta-i)\Lambda_0+(i-1)\Lambda_1))$ where $\Lambda_0$ and $\Lambda_1$ are the fundamental weights of $A_1^{(1)}$. 

As a notable development beyond the case $A_1^{(1)}$,  Capparelli \cite{Capparelli-thesis}  studied the structure of the level 3 standard modules of the affine Lie algebra $A_2^{(2)}$ and discovered two  partition identities. Nandi \cite{Nandi-thesis} conducted a similar investigation of the level 4 standard modules of $A_2^{(2)}$ and conjectured three partition identities modulo 14, which were subsequently proved by Takigiku and Tsuchioka~\cite{TT2024}. Baker, Kanade, Russell, and Sadowski~\cite{BKRS2023} obtained quadruple-sum representations for Nandi's identities. Using the Bailey pair machinery, Kanade and Russell \cite{KR-2022} presented sum-sides for principal characters of all integrable and highest weight irreducible modules for $A_2^{(2)}$. It is worth mentioning that the product sides of Nandi's identities \cite{Nandi-thesis} are also level 2 principal characters of $A_{11}^{(2)}$. Rogers--Ramanujan type identities with product side associated with the affine Lie
algebras $A_3^{(2)}$, $A_5^{(2)}$, $A_7^{(2)}$ and $A_9^{(2)}$ have also been studied in works such as \cite{Andrews2019,Bos-Misra,Kanade,KR-2019}. 

A closely related topic that has attracted a lot of attention in the last decade is Kanade--Russell conjectures \cite{KR-2015,KR-2019}.   Kanade and Russell \cite{KR-2015,KR-2019}  proposed some new partition identities and provided analytic sum representations. Some of their identities are motivated by level $2$ characters of $A_9^{(2)}$. Together with a  partition identity modulo 9 and two identities modulo 12 conjectured in  Russell's thesis \cite{Russell-thesis}, their conjectures contain seventeen identities modulo $12$ and five identities modulo $9$. Kur\c{s}ung\"oz~\cite{Kursungoz} constructed Andrews--Gordon type series for the partition generating functions for several Kanade--Russell conjectures. Tsuchioka \cite{Tsuchioka} reformulated three of the modulo $9$ identities in terms of vertex operators for level $3$ standard modules of the affine Lie algebra of type $D_4^{(3)}$. These Kanade--Russell identities have now been proved through works of Bringmann, Jennings-Shaffer, and Mahlburg \cite{BJSM}, Rosengren \cite{Rosengren},  Wang \cite{Wang2026}, Xia \cite{Xia2026} and Mizuno \cite{Mizuno2026}.

In this paper, we prove four conjectures of Takigiku and Tsuchioka concerning Rogers--Ramanujan type identities for level $2$ principal characters of $A_{2\ell-1}^{(2)}$. The first three, stated in \cite[Conjectures 6.1--6.3]{TT2021}, concern two sets of triple sum identities and one quadruple sum identity for $A_{13}^{(2)}$. The fourth, stated in \cite[Conjecture 1.3]{TT2026}, gives a family of multisum identities for certain level $2$ principal characters of $A_{6n+3}^{(2)}$. Let $\delta_{m,j}$ be the Kronecker delta function and $\mathcal{Q}$ denote the quintuple product \footnote{In \cite{TT2021} $\mathcal{Q}(x;q)$ is denoted as $Q(q,x)$.}
\begin{align}
\mathcal{Q}(x;q):=(-q/x,-x,q;q)_\infty(q/x^2,qx^2;q^2)_\infty, \quad x\neq 0.
\end{align}
As discussed in \cite{TT2026}, every level $2$ dominant integral weight of $A_{2\ell-1}^{(2)}$ can be written, up to an automorphism of the Dynkin diagram, as $\Lambda_m'=(\delta_{m,0}+\delta_{m,1})\Lambda_0+\Lambda_m$  for some $0\leq m \leq \ell$. For $1\leq m\leq\ell$, the principal character of the corresponding standard module is given by
\begin{align}
    \chi_{A_{2\ell-1}^{(2)}}(\Lambda_m')=\frac{(q^{2\ell+2};q^{2\ell+2})_\infty}{(q^2;q^2)_\infty}\frac{(q^{2m},q^{2\ell+2-2m};q^{2\ell+2})_\infty}{(q^m,q^{2\ell+2-m};q^{2\ell+2})_\infty}=\frac{\mathcal{Q}(q^m;q^{2\ell+2})}{(q^2;q^2)_\infty}.
\end{align}
In light of the preceding discussion and the work of Lepowsky and Wilson~\cite{LW1982,LW1984}, it is natural to seek an Eulerian multisum whose product evaluation is $\chi_{A_{2\ell-1}^{(2)}}(\Lambda_m')$.

In~\cite{TT2021}, Takigiku and Tsuchioka proved Andrews--Gordon type series for level $5$ and level $7$ $A_2^{(2)}$-modules and proposed further evaluations for level $2$ $A_{13}^{(2)}$-modules. Retaining their notation, we define
\begin{align}
    F_1(a,b,c):=\sum_{i,j,k\geq 0} (-1)^k\frac{q^{4\binom{i}{2}+2\binom{j}{2}+4\binom{k}{2}+2ij+4ik+4jk+ai+bj+ck}}{(q;q)_i(q^2;q^2)_j(q^4;q^4)_k}, \\
     F_2(a,b,c):=\sum_{i,j,k\geq 0} (-1)^j\frac{q^{\binom{i}{2}+2\binom{j}{2}+16\binom{k}{2}+2ij+4ik+4jk+ai+bj+ck}}{(q;q)_i(q^2;q^2)_j(q^4;q^4)_k}
\end{align}
and
\begin{align}
    F_3(a,b,c,d):=\sum_{i,j,k,\ell \geq 0}(-1)^k\frac{q^{2\binom{i}{2}+4\binom{j}{2}+2\binom{k}{2}+16\binom{\ell}{2}+2ij+2ik+4i\ell+4jk+8j\ell_+4k\ell +ai+bj+ck+d\ell}}{(q;q)_i(q^2;q^2)_j(q^2;q^2)_k(q^4;q^4)_\ell}.
\end{align}
Takigiku and Tsuchioka \cite[Conjectures 6.1--6.3]{TT2021} conjectured that with suitable choices of $a,b,c,d$, the functions $F_i$ ($i=1,2,3$) match with the character $\chi_{A_{13}^{(2)}}(\Lambda_m')$ with $m\in \{1,3,5,7\}$.
The first goal of this paper is to prove these three conjectures which we restate as the following theorems.
\begin{theorem}\label{thm-1}
(Cf.~\cite[Conjecture 6.1]{TT2021}.) 
We have
\begin{align}
    F_1(2,2,2)=\chi_{A_{13}^{(2)}}(\Lambda_3)=\frac{1}{(q^2,q^4,q^8,q^{12},q^{14},q^3,q^{13};q^{16})_\infty},  \label{thm1-id-1} \\
    F_1(4,2,6)=\chi_{A_{13}^{(2)}}(\Lambda_5)=\frac{1}{ (q^2,q^4,q^8,q^{12},q^{14},q^5,q^{11};q^{16})_\infty},  \label{thm1-id-2}\\
    F_1(6,4,6)=\chi_{A_{13}^{(2)}}(\Lambda_7)=\frac{1}{(q^4,q^6,q^8,q^{10},q^{12},q^7,q^{9};q^{16})_\infty}. \label{thm1-id-3}
\end{align}
\end{theorem}

\begin{theorem}\label{thm-2}
(Cf.~\cite[Conjecture 6.2]{TT2021}.) We have
\begin{align}
    F_2(1,3,12)=\chi_{A_{13}^{(2)}}(\Lambda_0+\Lambda_1)=\frac{1}{(q^4,q^6,q^8,q^{10},q^{12},q,q^{15};q^{16})_\infty}, \label{thm2-id-1} \\
    F_2(1,1,8)=\chi_{A_{13}^{(2)}}(\Lambda_3)=\frac{1}{(q^2,q^4,q^8,q^{12},q^{14},q^3,q^{13};q^{16})_\infty},  \label{thm2-id-2} \\
    F_2(3,3,6)=\chi_{A_{13}^{(2)}}(\Lambda_7)=\frac{1}{(q^4,q^6,q^8,q^{10},q^{12},q^7,q^{9};q^{16})_\infty}. \label{thm2-id-3}
\end{align}
\end{theorem}

\begin{theorem}\label{thm-3}
(Cf.~\cite[Conjecture 6.3]{TT2021}.) 
We have
\begin{align}
    F_3(1,5,1,12)=\chi_{A_{13}^{(2)}}(\Lambda_5)=\frac{1}{(q^2,q^4,q^8,q^{12},q^{14},q^5,q^{11};q^{16})_\infty}.  \label{thm3-id}
\end{align}
\end{theorem}

In 2026, Takigiku and Tsuchioka~\cite{TT2026} established families of identities for certain level $2$ characters of $A_{6n-1}^{(2)}$ and $A_{6n+1}^{(2)}$. These families recover several known partition identities, including the G\"ollnitz--Gordon  identities modulo $8$ (see, e.g.,~\cite[Sec.~2.4.4]{Sills-book}), the Rogers--Ramanujan  identities modulo $5$ in \eqref{RR}, and some of Nandi's identities modulo $14$  \cite{Nandi-thesis}. They \cite[Conjecture 1.3]{TT2026} further conjectured an analogous family for $A_{6n+3}^{(2)}$. The second aim of this paper is to prove this conjecture, which we restate below.

\begin{theorem}\label{thm-general}
(Cf.~\cite[Conjecture 1.3]{TT2026}) Let $n\geq 1$ and $0\le b\le2n$ be integers. We denote
$L_r=\ell_r+\cdots+\ell_{2n}$ for $1\le r\le2n$, and define
\begin{align}
 Q_n(i,j,k,\boldsymbol\ell)
 & =n(i+3j+3k)^2+2\sum_{r=1}^{2n}L_r^2
     +2(i+3j+3k)\sum_{r=1}^{2n}L_r+2iL_{2n},\label{eq-quadratic}\\
 R_{n,b}(i,j,k,\boldsymbol\ell)
 & =(b+2)i+(3b+2)j+(3b+4)k
     +2\sum_{r=2n+1-b}^{2n}L_r.\label{eq-linear}
\end{align}
We have
\begin{align}\label{eq-thm-general}
 &\sum_{i,j,k,\ell_1,\ldots,\ell_{2n}\geq 0}
 \frac{q^{Q_n(i,j,k,\boldsymbol\ell)+R_{n,b}(i,j,k,\boldsymbol\ell)}}
 {(q^2;q^2)_i(q^6;q^6)_j(q^6;q^6)_k
   \prod_{r=1}^{2n}(q^2;q^2)_{\ell_r}} \nonumber \\
&=\chi_{A_{6n+3}^{(2)}}(\Lambda_{n+2+b}')=\frac{\mathcal{Q}(q^{n+b+2};q^{6n+6})}{(q^2;q^2)_\infty}.
\end{align}
\end{theorem}
As pointed out in \cite{TT2026}, the cases
$(n,b)=(1,0),(1,2)$ recover two Kanade--Russell identities modulo 12  proved in \cite{BJSM,Rosengren}. 

The proofs of all the above theorems are based on Bailey pairs. A key clue that persuaded us to use this method is that the numerator in the character can be expanded as a theta series by the quintuple product identity as \cite[Theorem 1.3.17]{Berndt-book}
\begin{align}\label{eq-quintuple}
\mathcal{Q}(x;q)=\sum_{k\in\mathbb Z}(-1)^k q^{(3k^2-k)/2}
 (x^{3k}+x^{1-3k}).
\end{align}
This expansion allows the character to fit naturally into the $\alpha$-side of classical transformation formulas involving Bailey pairs (see, e.g.,~\eqref{BP-S2-a1-limit-negative} and \eqref{S4-aq-limit-negative}).  We then rewrite the multisums in equivalent forms so that they fit into the $\beta$-side and allow us to apply the Bailey pair method. The main difficult then lies in finding such equivalent expressions and establishing the corresponding Bailey pairs. In particular, the proof of Theorem~\ref{thm-general} draws on techniques used by Takigiku and Tsuchioka in their proof of the companion identity for $\chi_{A_{6n\pm 1}^{(2)}}(\Lambda_{n+b}')$ in \cite[Theorem 1.2]{TT2026}. However, our proof is essentially different from the proof strategy suggested by them \cite[Remark 3.1]{TT2026}.

The remainder of this paper is organized as follows. In Section \ref{sec-pre} we first review several basic $q$-series identities and relevant results from the theory of Bailey pairs, and establish a number of Bailey pairs needed for the proofs of our main theorems. We then prove Theorems~\ref{thm-1}--\ref{thm-3} in Section~\ref{sec-thm13} and the multisum identity in Theorem~\ref{thm-general} in Section~\ref{sec-thm-general}.

\section{Preliminaries}\label{sec-pre}
\subsection{Auxiliary $q$-series identities}
Recall Euler's $q$-exponential identities \cite[Corollary 1.3.2]{Berndt-book}
\begin{align}
\sum_{n=0}^\infty \frac{z^n}{(q;q)_n}&=\frac{1}{(z;q)_\infty}, \quad |z|<1,  \label{Euler}\\
\sum_{n=0}^\infty \frac{q^{\binom{n}{2}} z^n}{(q;q)_n}&=(-z;q)_\infty. \label{Euler2}
\end{align}
Both of them are corollaries of the $q$-binomial theorem \cite[Theorem 1.3.1]{Berndt-book}
\begin{align}\label{q-binomial}
    \sum_{n=0}^\infty \frac{(a;q)_n}{(q;q)_n}z^n=\frac{(az;q)_\infty}{(z;q)_\infty}, \quad |z|<1.
\end{align}

Let $(\alpha_n(a;q),\beta_n(a;q))$ be a Bailey pair relative to the parameter $a$ as defined in \eqref{BP-defn}. We shall review some useful consequences of Bailey's lemma (see, e.g., \cite{Bressoud2000}). First, we can construct new Bailey pairs $(\alpha_n',\beta_n')$ from a given Bailey pair $(\alpha_n,\beta_n)$ in various ways such as
\begin{equation}\label{S2}
\begin{split}
    \alpha_n'(a;q)&=a^{n/2}q^{n^2/2}\alpha_n(a;q), \\
    \beta_n'(a;q)&=\sum_{k=0}^n\frac{(-\sqrt{aq};q)_k}{(q;q)_{n-k}(-\sqrt{aq};q)_n}a^{k/2}q^{k^2/2}\beta_k(a;q);
\end{split} \tag{S2}
\end{equation}
\begin{equation}\label{S4}
    \begin{split}
        \alpha_n'(a;q)&=\frac{(-aq^{1/2};q)_n}{(-q^{1/2};q)_n}q^{n^2/2}\alpha_n(a;q), \\
        \beta_n'(a;q)&=\sum_{k=0}^n\frac{(-aq^{1/2};q)_k}{(q;q)_{n-k}(-q^{1/2};q)_n}q^{k^2/2}\beta_k(a;q).
    \end{split}\tag{S4}
\end{equation}
If we set $a=1$ in \eqref{S2}, then we obtain a Bailey pair
\begin{equation}\label{BP-S2-a1}
\begin{split}
    \alpha_n'(1;q)&=q^{n^2/2}\alpha_n(1;q), \\
    \beta_n'(1;q)&=\sum_{k=0}^n \frac{(-q^{1/2};q)_kq^{k^2/2}}{(q;q)_{n-k}(-q^{1/2};q)_n}\beta_k(1;q).
\end{split}
\end{equation}
Let $n\rightarrow \infty $. We obtain the transformation formula
\begin{align}\label{BP-S2-a1-limit}
    \sum_{n=0}^\infty (-q^{1/2};q)_nq^{n^2/2}\beta_n(1;q)=\frac{(-q^{1/2};q)_\infty}{(q;q)_\infty}\sum_{n=0}^\infty q^{n^2/2}\alpha_n(1;q).
\end{align}
Replacing $q$ by $q^2$ and then replacing $q$ by $-q$, we obtain
\begin{align}\label{BP-S2-a1-limit-negative}
    \sum_{n=0}^\infty (-1)^nq^{n^2} (q;q^2)_n\beta_n(1;q^2)=\frac{(q;q^2)_\infty}{(q^2;q^2)_\infty}\sum_{n=0}^\infty (-1)^nq^{n^2}\alpha_n(1;q^2).
\end{align}
Similarly, if we set $a=q$ in \eqref{S2},  and then let $n\rightarrow \infty $, we obtain the formula
\begin{align}\label{BP-S2-aq-limit}
    \sum_{n=0}^\infty (-q;q)_nq^{n(n+1)/2}\beta_n(q;q)=\frac{(-q;q)_\infty}{(q^2;q)_\infty}\sum_{n=0}^\infty q^{n(n+1)/2}\alpha_n(q;q).
\end{align}

If we set $a=q$ in \eqref{S4} and then letting $n\rightarrow \infty$, we obtain
\begin{align}\label{S4-aq-infinite}
    \sum_{n=0}^\infty q^{\frac{1}{2}n^2}(-q^{\frac{3}{2}};q)_n\beta_n(q;q)=\frac{(-q^{\frac{1}{2}};q)_\infty}{(q^2;q)_\infty} \sum_{n=0}^\infty \frac{1+q^{n+\frac{1}{2}}}{1+q^{\frac{1}{2}}}q^{\frac{1}{2}n^2}\alpha_n(q;q).
\end{align}
Replacing $q$ by $q^2$ and then replacing $q$ by $-q$, we obtain
\begin{align}\label{S4-aq-limit-negative}
    \sum_{n=0}^\infty (-1)^nq^{n^2}(q^{3};q^2)_n\beta_n(q^2;q^2)=\frac{(q;q^2)_\infty}{(q^4;q^2)_\infty} \sum_{n=0}^\infty \frac{1-q^{2n+1}}{1-q}(-1)^nq^{n^2}\alpha_n(q^2;q^2).
\end{align}

Similarly, from \cite[(S6)]{Bressoud2000} we deduce the following transformation formula 
\begin{align}\label{S6-a1-limit}
    \sum_{n=0}^\infty q^{(n^2+n)/2}(-1;q)_n\beta_n(1;q)=\frac{(-q;q)_\infty}{(q;q)_\infty}\sum_{n=0}^\infty \frac{(-1;q)_n}{(-q;q)_n}q^{(n^2+n)/2}\alpha_n(1;q).
\end{align}
The above transformation formulas played a key role in proving Theorems \ref{thm-1}--\ref{thm-3}. As for the proof of Theorem \ref{thm-general}, we will use the following form of the Bailey lattice which was also used in the work of Takigiku and Tsuchioka \cite[Theorem~2.4]{TT2026}.
\begin{lemma}\label{lem-lattice}
(Cf.~\cite[Corollary~4.2]{AAB}.) Suppose $(\alpha_r,\beta_r)_{r\geq 0}$ is a Bailey pair relative to $a$. For integers $k\geq 1$ and $0\leq j\leq k$, we have
\begin{align}
 &\sum_{N_1\geq\cdots\geq N_k\geq 0}
 \frac{q^{N_1^2+\cdots+N_k^2-(N_1+\cdots+N_j)}
       a^{N_1+\cdots+N_k}}
 {(q;q)_{N_1-N_2}\cdots(q;q)_{N_{k-1}-N_k}}\beta_{N_k}\nonumber\\
 &=
 \frac{1}{(aq;q)_\infty}
 \sum_{r\geq 0}a^{kr}q^{kr^2-jr}
 \frac{1-a^{j+1}q^{2r(j+1)}}{1-aq^{2r}}\alpha_r.
 \label{eq-lattice}
\end{align}
\end{lemma}

\subsection{Finite sum identities involving Gaussian polynomials} For integers $N,r$, the $q$-binomial coefficient or Gaussian polynomial is defined as:
\begin{align}
{N \brack r}={N \brack r}_{q}=
 \begin{cases}
\frac{(q;q)_N}{(q;q)_r(q;q)_{N-r}}, & 0\leq r\leq N,\\ 
 0,&\text{otherwise}.
 \end{cases}
\end{align}
It is well known that \cite[Sec.~1.1]{Berndt-book}
\begin{align}
 {N \brack r}_{q}&={N \brack N-r}_{q},\label{qbinomial-symmetry}\\
 {N \brack r}_{q}&={N-1 \brack r}_{q}+q^{N-r}{N-1 \brack r-1}_{q},\label{qbinomial-1}\\
 {N \brack r}_{q}&=q^r {N-1 \brack r}_{q}+{N-1 \brack r-1}_{q}.  \label{qbinomial-2}
\end{align}

For $L,M\in\mathbb{Z}_{\geq 0}$ and $a,b\in\mathbb{Z}$, we recall the following functions from Warnaar's work \cite[(2.10) and (2.25)]{Warnaar2003}:
\begin{align}
\mathscr{S}(L,M,a,b;q)&=\sum_{s=0}^{L}q^{s(s+a)}
 {L+M-a-2s \brack M}_{q}{M-a+b \brack s}_{q}{M+a-b \brack s+a}_{q},\label{eq-Sdef}\\
 \mathscr{V}(L,M,a,b;q)&=\mathscr{S}(L,M,a,b;q)+q^{b+1/2}\mathscr{S}(L,M,a+1,b+1;q).
 \label{eq-Vdef}
\end{align}

The last two Gaussian factors in \eqref{eq-Sdef} imply
\begin{equation}\label{eq-S-zero}
 \mathscr{S}(L,M,a,b;q)=0  \quad\text{if}\quad |b|>M\ \text{or}\ |a-b|>M.
\end{equation}

From \cite[(2.17)]{Warnaar2003} we know that if $M\in\mathbb{Z}_{\geq 0}$ and $a,b\in\mathbb{Z}$ satisfy $|a-b|\le M$, then
\begin{equation}\label{eq-S-limit}
 (q;q)_M\lim_{L\to\infty}\mathscr{S}(L,M,a,b;q)={2M \brack M-b}_{q}. 
\end{equation}

Next, we define an auxiliary polynomial by
\begin{align}\label{eq-Pdef}
 &P_n(x):=\sum_{i=0}^{n}q^{i(i-1)}(-q;q)_i {n \brack i}_{q^2}x^i=(q^2;q^2)_n\sum_{i=0}^{n}
 \frac{q^{i(i-1)}x^i}{(q;q)_i(q^2;q^2)_{n-i}}.
\end{align}

\begin{lemma}\label{lem-P-1}
For every $n\in\mathbb{Z}_{\geq 0}$,
\begin{align}
 P_n(q)&=\sum_{j\in\mathbb{Z}}(-1)^j \left\{q^{15j^2+2j} {2n \brack n-3j}_{q^2}
       +q^{15j^2+8j+1} {2n \brack n-3j-1}_{q^2}\right\},  \label{eq-W1}\\
 P_n(q^2)&=\sum_{j\in\mathbb{Z}}(-1)^j \left\{q^{15j^2+4j} {2n \brack n-3j}_{q^2}
       +q^{15j^2+14j+3}  {2n \brack n-3j-1}_{q^2}\right\}.  \label{eq-W2}
\end{align}
\end{lemma}
\begin{proof}
From \cite[p.~242]{Warnaar2003} we find 
\begin{align}
   & \sum_{j\in\mathbb{Z}}(-1)^j \left\{q^{15j^2+2j} {2n \brack n-3j}_{q^2}
       +q^{15j^2+8j+1} {2n \brack n-3j-1}_{q^2}\right\} \nonumber \\
&=\sum_{i=0}^n q^{i^2}{ n\brack i}(-q^{n-i+1};q)_i.
\end{align}
This proves \eqref{eq-W1} after rewriting the right side as $P_n(q)$. 

From \cite[(5.11)]{Warnaar2003} we have
\begin{align}
 &\sum_{j\in\mathbb{Z}}(-1)^jq^{15j^2+4j} \Big\{\mathscr{S}(L,n,5j+1,3j;q^2)
 +q^{10j+3}\mathscr{S}(L,n,5j+2,3j+1;q^2)\Big\}\nonumber\\
 &\quad=\sum_{i=0}^{n}q^{i(i+1)} {L+n-i  \brack L}_{q^2}  {2L-i-1 \brack i}_{q}.
 \label{eq-Warnaar511}
\end{align}
Multiplying by $(q^2;q^2)_n$ and let $L\to\infty$, using \eqref{eq-S-limit} with $q$ replaced by $q^2$, we see that its left side agrees with the right side in \eqref{eq-W2}. 

On the unilateral sides, as $L\to\infty$ we have
\begin{align*}
{L+n-i \brack L}_{q^2}\longrightarrow\frac{1}{(q^2;q^2)_{n-i}},\quad
 {2L-i-1 \brack i}_{q}\longrightarrow\frac{1}{(q;q)_i}.
\end{align*}
Thus the limit of the right side of \eqref{eq-Warnaar511} is $P_n(q^2)$. This proves \eqref{eq-W2}.
\end{proof}

\begin{lemma}\label{lem-P-2}
For every $n\in\mathbb{Z}_{\geq 0}$,
\begin{align}
 (1+q^{2n})P_n(1) &=\sum_{j\in\mathbb{Z}}(-1)^j
 \Biggl\{ \big(q^{15j^2+2j}+q^{15j^2+4j}\big) {2n \brack n-3j}_{q^2}
 \nonumber\\
 & \quad +\big(q^{15j^2+8j+1}+q^{15j^2+14j+3}\big){2n \brack n-3j-1}_{q^2}
 \Biggr\},\label{eq-finite3}\\
 P_n(q^2)
 &=\sum_{j\in\mathbb{Z}}(-1)^j \left\{q^{15j^2+4j} {2n+1 \brack n-3j}_{q^2} +q^{15j^2+14j+3} {2n+1 \brack n-3j-1}_{q^2} \right\},\label{eq-finite5}\\
 q^{2n}P_n(q^2)
 &=\sum_{j\in\mathbb{Z}}(-1)^jq^{15j^2+10j} \left\{ {2n+1 \brack n-3j}_{q^2}
       -{2n+1 \brack n-3j-1}_{q^2}\right\}.
 \label{eq-finite7}
\end{align}
\end{lemma}
\begin{proof}
(1) The case $n=0$ is trivial. For $n\geq 1$, from \eqref{eq-Pdef}  we have
\begin{align}
 P_n(1)-P_n(q)=(q^2;q^2)_n\sum_{i=0}^{n-1}\frac{q^{i(i+1)}}{(q;q)_i(q^2;q^2)_{n-1-i}},\label{eq-Pdiff1}\\
 P_n(q^2)-q^{2n}P_n(1)=(q^2;q^2)_n\sum_{i=0}^{n-1}\frac{q^{i(i+1)}}{(q;q)_i(q^2;q^2)_{n-1-i}}.
 \label{eq-Pdiff2}
\end{align}
The right side of \eqref{eq-Pdiff1} and \eqref{eq-Pdiff2} are the same, and hence their left side also agree with each other. This implies
\begin{equation}\label{eq-P-relation}
 P_n(q)+P_n(q^2)=(1+q^{2n})P_n(1).  
\end{equation}
Adding \eqref{eq-W1} and \eqref{eq-W2} up, we obtain \eqref{eq-finite3}.

(2) Let $R_n$ denote its right-hand side of \eqref{eq-finite5}. Applying \eqref{qbinomial-1} to both Gaussian polynomial in \eqref{eq-W2}, we deduce that
\begin{align}
    P_n(q^2)&=\sum_{j\in \mathbb{Z}} (-1)^j \Big\{ q^{15j^2+4j}\Big( {2n+1 \brack n-3j}_{q^2}-q^{2(n+3j+1)}{2n\brack n-3j-1}_{q^2}\Big) \nonumber \\ &\quad +q^{15j^2+14j+3}\Big({2n+1 \brack n-3j-1}_{q^2}-q^{2(n+3j+2)}{2n \brack n-3j-2}_{q^2}\Big)\Big\} \nonumber \\
    &=R_n-\sum_{j\in \mathbb{Z}} (-1)^j \Big\{ q^{15j^2+10j+2n+2}{2n \brack n-3j-1}_{q^2}+q^{15j^2+20j+2n+7}{2n \brack n-3j-2}_{q^2}\Big\}. \nonumber 
\end{align}
Replacing $j$ by $-j-1$ in the last Gaussian polynomial, using \eqref{qbinomial-symmetry} and noting that
\begin{align*}
 15(-j-1)^2+20(-j-1)+2n+7=15j^2+10j+2n+2,
\end{align*}
we see that the two sums cancel with each other. This proves  \eqref{eq-finite5}.

(3) Using \eqref{qbinomial-2} for ${2n+1 \brack r}_{q^2}$ and
\eqref{qbinomial-1} for ${2n+1 \brack r-1}_{q^2}$, we see that their difference is
\begin{equation}\label{eq-Gauss-diff}
 {2n+1 \brack r}_{q^2}-{2n+1 \brack r-1}_{q^2}=q^{2r}{2n \brack r}_{q^2}-q^{4n+4-2r}{ 2n \brack r-2}_{q^2}.  
\end{equation}
If $T_n$ denotes the right-hand side of \eqref{eq-finite7},
substitution of \eqref{eq-Gauss-diff} with $r=n-3j$ yields
\begin{equation}\label{Tn-mid}
 T_n=q^{2n}\sum_{j\in\mathbb{Z}}(-1)^j\left\{q^{15j^2+4j}{2n \brack n-3j}_{q^2}
       -q^{15j^2+16j+4}{ 2n \brack n-3j-2}_{q^2}\right\}.
\end{equation}
Replacing $j$ by $-j-1$ for the second sum and noting that
\begin{align*}
 15(-j-1)^2+16(-j-1)+4=15j^2+14j+3,
\end{align*}
we conclude from \eqref{eq-W2} that 
\begin{align*}
 T_n&=q^{2n}\sum_{j\in\mathbb{Z}}(-1)^j\left\{q^{15j^2+4j}{2n \brack n-3j}_{q^2}
       +q^{15j^2+14j+3}{ 2n \brack n-3j-1}_{q^2}\right\}=q^{2n}P_n(q^2).
\end{align*}
This proves \eqref{eq-finite7}.
\end{proof}

\begin{corollary}\label{cor-BP123}
Let $\alpha_0^{(i)}=1$ ($i=1,2,3$) and for $n\geq 1$ we define
\begin{equation}\label{BP1}
\begin{split}
\alpha_n^{(1)}(1;q^2)&=\begin{cases}
\frac{(-1)^m}{2}q^{15m^2-4m}(1+q^{2m})(1+q^{6m}), & n=3m, \\
\frac{(-1)^m}{2}q^{15m^2+8m+1}(1+q^{6m+2}), & n=3m+1, \\
\frac{(-1)^{m+1}}{2}q^{15m^2+16m+4}(1+q^{6m+4}), & n=3m+2,
\end{cases} \\
\beta_n^{(1)}(1;q^2)&=\frac{1+q^{2n}}{2}\frac{P_n(1)}{(q^2;q^2)_{2n}};
\end{split}
\end{equation}
\begin{equation}\label{BP2}
\begin{split}
 \alpha_{n}^{(2)}(q^2;q^2)&=\begin{cases}
 (-1)^m q^{15m^2+4m}, &n=3m,    \\
(-1)^m q^{15m^2-4m}, &n=3m-1,    \\
 (-1)^m q^{15m^2+14m+3}(1-q^{2m+1}), &n=3m+1,
\end{cases} \\
 \beta_n^{(2)}(q^2;q^2)&=\frac{P_n(q^2)}{(q^4;q^2)_{2n}};
\end{split}
\end{equation}
\begin{equation}\label{BP3}
\begin{split}
\alpha_n^{(3)}(q^2;q^2)&=\begin{cases}
(-1)^m q^{15m^2+10m},  &n=3m,     \\
(-1)^m q^{15m^2-10m},  &n=3m-1,  \\
(-1)^{m+1}q^{15m^2+10m}(1-q^{10m+5}), &n=3m+1, 
\end{cases} \\
 \beta_n^{(3)}(q^2;q^2)&=q^{2n}\frac{P_n(q^2)}{(q^4;q^2)_{2n}}.  
\end{split}
\end{equation}
Then the above sequences $(\alpha_n^{(i)}(a;q^2),\beta_n^{(i)}(a;q^2))$ ($a=1,q^2$, $i=1,2,3$) form Bailey pairs relative to $a$ with base $q^2$.
\end{corollary}
\begin{proof}
This follows immediately from Lemma \ref{lem-P-2} and the definition \eqref{BP-defn}.
\end{proof}

For $\sigma\in\{0,1\}$, we define
\begin{align}
 D_n^{(\sigma)}(q):=(q;q)_n\sum_{r=0}^{\lfloor n/2\rfloor}
 \frac{q^{2r(r+\sigma)}}{(q;q)_{n-2r}(q^2;q^2)_r}=\sum_{r=0}^{\lfloor n/2\rfloor}
 q^{2r(r+\sigma)}(q;q^2)_r {n \brack 2r}_{q}.
 \label{eq-Ddef}
\end{align}

\begin{lemma}\label{lem-baseq}
For every $n\in\mathbb{Z}_{\geq 0}$,
\begin{align}
 D_n^{(0)}(q)&=\sum_{j\in\mathbb{Z}}q^{(15j^2+j)/2}
 \left\{ {2n \brack n-3j}_{q}-q^{5j+1}{2n \brack n-3j-1}_{q}\right\},
 \label{Dn0-id}\\
 D_n^{(1)}(q)
 &=\sum_{j\in\mathbb{Z}}q^{(15j^2+7j)/2} \left\{{2n+1 \brack n-3j}_{q}
       -q^{3j+1} {2n+1 \brack n-3j-1}_{q}\right\}. \label{Dn1-id}
\end{align}
\end{lemma}
\begin{proof}
We denote two integers
\begin{equation}\label{eq-alphabeta}
 \alpha_j=\frac{15j^2+j}{2},\quad  \beta_j=\frac{(3j+1)(5j+2)}{2}=\alpha_j+5j+1.
\end{equation}
We write the last identity of
\cite[\S5.2, p.~245]{Warnaar2003} with parameter $k=2$ as
\begin{align}
 &\sum_{j\in\mathbb{Z}} \Big(q^{\alpha_j}\mathscr{V}(L,M,5j,3j;q)
       -q^{\beta_j}\mathscr{V}(L,M,5j+2,3j+1;q)\Big)\nonumber\\
 &\quad=\sum_{s=0}^{\min(L,M)}q^{s^2/2}{L+M-s \brack L}_{q}
 {\lfloor L-s/2\rfloor \brack L-s}_{q^2}=:S_0(q)+q^{1/2}S_1(q),
 \label{eq-Warnaar}
\end{align}
where $S_0(q)$ and $S_1(q)$ are power series in $q$. Let $\sigma\in\{0,1\}$. We now evaluate $S_0(q)$ and $S_1(q)$ explicitly. For the left side we use \eqref{eq-Vdef}. On the right of \eqref{eq-Warnaar}, only $s=2r+\sigma$ contributes to $S_\sigma(q)$. Since
\begin{align*}
 \frac{(2r+\sigma)^2-\sigma}{2}=2r(r+\sigma),\quad
 \left\lfloor L-\frac{2r+\sigma}{2}\right\rfloor=L-r-\sigma,
\end{align*}
we deduce that
\begin{align}
 q^{\sigma/2} S_\sigma(q)=&\sum_{j\in\mathbb{Z}}\big\{
 q^{\alpha_j+3\sigma j}\mathscr{S}(L,M,5j+\sigma,3j+\sigma;q)\nonumber\\ 
 &\quad -q^{\beta_j+\sigma(3j+1)}\mathscr{S}(L,M,5j+2+\sigma,3j+1+\sigma;q)\big\}\nonumber\\
 &=\sum_{r=0}^{\lfloor(\min(L,M)-\sigma)/2\rfloor}
 q^{2r(r+\sigma)}{L+M-2r-\sigma \brack L}_{q}{L-r-\sigma \brack r}_{q^2}.
 \label{eq-S-sigma}
\end{align}

For the four parameter pairs in \eqref{eq-S-sigma}, as $\sigma$
ranges over $0,1$, one has $|a-b|\le |b|$. Thus every term with
$|b|\le M$ satisfies the hypothesis of \eqref{eq-S-limit}.
Multiply \eqref{eq-S-sigma} by $(q;q)_M$ and let $L\to\infty$.
For each fixed  $r$,
\begin{align*}
 {L+M-2r-\sigma \brack L}_{q}\longrightarrow  \frac{1}{(q;q)_{M-2r-\sigma}},\quad
 {L-r-\sigma \brack r}_{q^2}\longrightarrow  \frac{1}{(q^2;q^2)_r}.
\end{align*}
Consequently, for $M\geq \sigma$,
\begin{align}
 &\sum_{j\in\mathbb{Z}}\Big\{q^{\alpha_j+3\sigma j} {2M \brack M-3j-\sigma}_{q}
 -q^{\beta_j+\sigma(3j+1)} {2M \brack M-3j-1-\sigma}_{q}\Big\}\nonumber\\
 &\quad=(q;q)_M\sum_{r=0}^{\lfloor(M-\sigma)/2\rfloor}\frac{q^{2r(r+\sigma)}}{(q;q)_{M-2r-\sigma}(q^2;q^2)_r}.
 \label{eq-paritylimit}
\end{align}

For $\sigma=0$, setting $M=n$ in \eqref{eq-paritylimit} and using $\beta_j=\alpha_j+5j+1$, we prove \eqref{Dn0-id}.

For $\sigma=1$, we denote
\begin{align*}
 \gamma_j=\alpha_j+3j=\frac{15j^2+7j}{2},\quad \delta_j=\beta_j+3j+1=\frac{15j^2+17j+4}{2}.
\end{align*}
In the second bilateral term of \eqref{eq-paritylimit}, replace $j$ by $-j-1$. Then
\begin{align*}
 \delta_{-j-1}=\gamma_j+3j+1,\quad {2M \brack M+3j+1}_{q}={2M \brack M-3j-1}_{q}.
\end{align*}
Hence the odd component becomes
\begin{align} \label{eq-odd-rewrite}
\sum_{j\in\mathbb{Z}}q^{\gamma_j}(1-q^{3j+1}){2M \brack M-3j-1}_{q}=(q;q)_M\sum_{r=0}^{\lfloor(M-1)/2\rfloor}\frac{q^{2r(r+1)}}{(q;q)_{M-2r-1}(q^2;q^2)_r}.
\end{align}

The change to an odd upper index uses the elementary relation
\begin{align} \label{eq-add-id}
(1-q^{n+1}) \left\{{2n+1 \brack n-3j}_{q}-q^{3j+1} {2n+1 \brack n-3j-1}_{q}\right\}=(1-q^{3j+1}){2n+2 \brack n-3j}_{q}.
\end{align}
To verify it, write $d=n-3j$ and use
\begin{align*}   
{2n+1 \brack d}_{q}=\frac{1-q^{2n+2-d}}{1-q^{2n+2}}{2n+2 \brack d}_{q},\quad
 {2n+1 \brack d-1}_{q}=\frac{1-q^d}{1-q^{2n+2}} {2n+2 \brack d}_{q}.
\end{align*}
The numerator of the required difference factors as
\begin{align*}
 (1-q^{2n+2-d})-q^{n+1-d}(1-q^d)=(1+q^{n+1})(1-q^{n+1-d}).
\end{align*}
Canceling $1-q^{2n+2}=(1-q^{n+1})(1+q^{n+1})$ proves \eqref{eq-add-id}.

Finally, take $M=n+1$ in \eqref{eq-odd-rewrite} and apply \eqref{eq-add-id}. Both sides contain $1-q^{n+1}$, and canceling it gives \eqref{Dn1-id}. This also covers $n=0$ by using $M=1$. This proves Lemma~\ref{lem-baseq}.
\end{proof}

\begin{corollary}\label{cor-BP45}
Let $\alpha_0^{(i)}=1$ ($i=4,5$) and for $n\geq 1$ we define
\begin{equation}\label{BP4}
\alpha_n^{(4)}(1;q)=\begin{cases}
    q^{(15m^2-m)/2}(1+q^m) & n=3m, \\
    -q^{\frac{1}{2}(15m^2+11m)+1} & n=3m+1, \\
    -q^{\frac{1}{2}(15m^2+19m)+3} & n=3m+2,
\end{cases} \quad    \beta_n^{(4)}(1;q)=\frac{D_n^{(0)}(q)}{(q;q)_{2n}};
\end{equation}
\begin{equation}\label{BP5}
\alpha_n^{(5)}(1;q)=\begin{cases}
 q^{(15m^2-5m)/2}(1+q^{5m}) & n=3m,\\
 -q^{\frac{1}{2}(15m^2+5m)} & n=3m+1, \\
 -q^{\frac{1}{2}(15m^2+25m)+5} & n=3m+2,
\end{cases} \quad \beta_n^{(5)}(1;q)=\frac{q^nD_n^{(0)}(q)}{(q;q)_{2n}}. 
\end{equation}
Then $(\alpha_n^{(i)},\beta_n^{(i)})$ ($i=4,5)$ are Bailey pairs relative to $1$.
\end{corollary}
\begin{proof}
(1) Replacing $k$ by $-k$ or $-k-1$, we can rewrite \eqref{Dn0-id}  as
\begin{align}
D_n^{(0)}&={2n\brack n}+\sum_{j\geq 1}
 {2n \brack n-3j}_{q}\Big(q^{(15j^2+j)/2}+q^{(15j^2-j)/2}\Big)
      \nonumber \\
    & \quad -\sum_{j\geq 0} q^{(15j^2+j)/2+5j+1}{2n \brack n-3j-1}_{q}-\sum_{j\geq 0} q^{(15j^2+19j)/2+3} {2n \brack n-3j-2}.
\end{align}
After dividing both sides by $(q;q)_{2n}$, we obtain the Bailey pair $(\alpha_n^{(4)},\beta_n^{(4)})$.

(2) Replacing $j$ by $-j$ in the second Gaussian polynomial in \eqref{Dn0-id}, we may rewrite it as 
\begin{align}
    D_n^{(0)}(q)=\sum_{j\in\mathbb{Z}}q^{(15j^2+j)/2}\left\{ {2n \brack n-3j}_{q}
       -q^{-6j+1}{2n \brack n-3j+1}_{q}\right\}.
 \label{Dn0-id-rewrite}
\end{align}
It is easy to see that
\begin{align}
(1-q^{n+s}){2n \brack n-s}=(1-q^{n-s+1}){2n \brack n-s+1}.
\end{align}
Equivalently, we have
\begin{align}\label{binomial-key}
    q^n\Big( q^s{2n \brack n-s}-q^{1-s}{2n \brack n-s+1}\Big)={2n \brack n-s}-{2n \brack n-s+1}.
\end{align}
Applying \eqref{binomial-key} with $s=3j$ to \eqref{Dn0-id-rewrite}, we obtain
\begin{align}
q^nD_n^{(0)}(q)=\sum_{j\in \mathbb{Z}} q^{(15j^2-5j)/2}\Big({2n\brack n-3j}-{2n \brack n-3j+1}\Big).
\end{align}
After replacing $j$ by $-j$ or $j+1$, we obtain the Bailey pair $(\alpha_n^{(5)}(1;q),\beta_n^{(5)}(1;q))$.
\end{proof}

\begin{corollary}\label{cor-BP67}
Let 
\begin{equation}
 A_m(q):=\frac{q^{(15m^2+13m)/2}}{1+q^{3m+1}}, \quad B_m(q):=\frac{q^{(15m^2+23m+6)/2}}{1+q^{3m+2}}.
 \label{eq-S}
\end{equation}
We have the following Bailey pairs relative to $q$:
\begin{align}\label{BP6}
&\alpha_n^{(6)}(q;q)=\begin{cases}
   q^{(15m^2+7m)/2}, & n=3m \\
   q^{(15m^2-7m)/2}, & n=3m-1, \\
   -q^{(15m^2+13m+2)/2}(1+q^{2m+1}), & n=3m+1,
\end{cases} \quad \beta_n^{(6)}(q;q)=\frac{D_n^{(1)}(q)}{(q^2;q)_{2n}}; \\
&\alpha_n^{(7)}(q;q)=\begin{cases}
    A_m(q), & n=3m, \\
    -A_m(q)-B_m(q), & n=3m+1, \\
    B_m(q), & n=3m+2,
    \end{cases} \quad \beta_n^{(7)}(q;q)=\frac{q^n}{1+q^{n+1}}\frac{D_n^{(1)}(q)}{(q^2;q)_{2n}}. \label{BP7}
\end{align}
\end{corollary}
\begin{proof}
(1) The first Bailey pair follows directly from \eqref{Dn1-id} after replacing $j$ by $-j$ or $-j-1$.

(2) We denote
\begin{align}\label{Kn-defn}
    K_n(r)=\frac{1}{(q;q)_{2n+1}}{2n+1 \brack n-r}=\frac{1}{(q;q)_{n-r}(q;q)_{n+r+1}}
\end{align}
for $-n-1\leq r\leq n$ and $K_n(r)=0$ otherwise. It is obvious that
\begin{align}\label{Kn-symmetry}
    K_n(r)=K_n(-r-1).
\end{align}
It is straightforward to calculate that
\begin{align*}
   K_n(r)-q^{r+1}K_n(r+1)=\frac{(1-q^{r+1})(1+q^{n+1})}{(q;q)_{n-r}(q;q)_{n+r+2}}, \\
   K_n(r)-K_n(r+1)=\frac{q^{n-r}(1-q^{r+1})(1+q^{r+1})}{(q;q)_{n-r}(q;q)_{n+r+2}}.
\end{align*}
Hence we have
\begin{align}\label{Kn-diff}
    \frac{q^n}{1+q^{n+1}}\big(K_n(r)-q^{r+1}K_n(r+1)\big)=\frac{q^r}{1+q^{r+1}}\big(K_n(r)-K_n(r+1)\big).
\end{align}
Applying \eqref{Kn-diff} with $r=3j$ to \eqref{Dn1-id}, we deduce that
\begin{align}\label{beta6-start}
    \beta_n^{(6)}(q;q)&=\frac{q^n}{1+q^{n+1}}\sum_{j\in \mathbb{Z}}q^{(15j^2+7j)/2}\big(K_n(3j)-q^{3j+1}K_n(3j+1)\big) \nonumber \\
    &=\sum_{j\in \mathbb{Z}} \frac{q^{(15j^2+13j)/2}}{1+q^{3j+1}}\big(K_n(3j)-K_n(3j+1)\big).
\end{align}
Now split the sum into $j=m$ and $j=-m-1$ with $m\geq 0$. Using \eqref{Kn-symmetry} we can rewrite \eqref{beta6-start} as
\begin{align}
    \beta_n^{(6)}(q;q)=\sum_{m\geq 0} \big( A_m K_n(3m)-(A_m+B_m)K_n(3m+1)+B_mK_n(3m+2)\big).
\end{align}
From it we conclude that $(\alpha_n^{(7)},\beta_n^{(7)})$ is a Bailey pair relative to $q$. 
\end{proof}

\section{Proofs of the triple and quadruple sums identities}\label{sec-thm13}
We are now ready to give proofs for Theorems \ref{thm-1}--\ref{thm-3}.
\begin{proof}[Proof of Theorem \ref{thm-1}]
Summing over $k$ using \eqref{Euler}, we deduce that
\begin{align}
    F_1(a,b,c)=\sum_{i,j\geq 0} \frac{q^{(i+j)^2+i^2+(a-2)i+(b-1)j}}{(q;q)_i(q^2;q^2)_j}(q^{4i+4j+c};q^4)_\infty.
\end{align}
After writing $i+j=n$, we have 
\begin{align}
F_1(2,2,2)&=(q^2;q^4)_\infty \sum_{n=0}^\infty \frac{q^{n^2+n}}{(q^2;q^4)_n}\sum_{i=0}^n \frac{q^{i^2-i}}{(q;q)_i(q^2;q^2)_{n-i}}, \label{F1-id-1-proof-start} \\
  F_1(4,2,6)&=(q^2;q^4)_\infty \sum_{n=0}^\infty \frac{q^{n^2+n}}{(q^2;q^4)_{n+1}}\sum_{i=0}^n \frac{q^{i^2+i}}{(q;q)_i(q^2;q^2)_{n-i}}, \label{F1-id-2-proof-start} \\
    F_1(6,4,6)&=(q^2;q^4)_\infty \sum_{n=0}^\infty \frac{q^{n^2+3n}}{(q^2;q^4)_{n+1}}\sum_{i=0}^n \frac{q^{i^2+i}}{(q;q)_i(q^2;q^2)_{n-i}}. \label{F1-id-3-proof-start}
\end{align}

Recall the function $P_n(x)$ defined in \eqref{eq-Pdef} and the Bailey pairs $(\alpha_n^{(i)},\beta_n^{(i)})$ in Corollary \ref{cor-BP123}. We have
\begin{align}
F_1(2,2,2)&=(q^2;q^4)_\infty \sum_{n=0}^\infty \frac{q^{n^2+n}}{(q^2;q^4)_n} \frac{P_n(1)}{(q^2;q^2)_n}  \nonumber \\
&=(q^2;q^4)_\infty \sum_{n=0}^\infty q^{n^2+n}(-q^2;q^2)_n \frac{P_n(1)}{(q^2;q^2)_{2n}} \nonumber \\
&=(q^2;q^4)_\infty \sum_{n=0}^\infty q^{n^2+n}(-1;q^2)_n \beta_n^{(1)}(1;q^2) \nonumber \\
&= \frac{2}{(q^2;q^2)_\infty} \sum_{n\geq 0}\frac{q^{n(n+1)}\alpha_n^{(1)}(1;q^2)}{1+q^{2n}} \quad \text{(by \eqref{S6-a1-limit})} \nonumber \\
 &=\frac{2}{(q^2;q^2)_\infty} \sum_{k\in\mathbb{Z}}(-1)^k
 \left(q^{24k^2+k}+q^{24k^2-17k+3}\right) \nonumber \\
 &=\frac{2}{(q^2;q^2)_\infty}  \mathcal Q(q^3;q^{16}).
\end{align}
Here for the last equality we used \eqref{eq-quintuple}. This proves \eqref{thm1-id-1}.

Similarly, we have
\begin{align}
F_1(4,2,6)&=(q^2;q^4)_\infty \sum_{n=0}^\infty \frac{q^{n^2+n}}{(q^2;q^4)_{n+1}} \frac{P_n(q^2)}{(q^2;q^2)_n} \nonumber \\
&=\frac{(q^2;q^4)_\infty}{1-q^2} \sum_{n=0}^\infty q^{n^2+n}(-q^2;q^2)_n \beta_n^{(2)}(q^2;q^2) \nonumber \\
&=\frac{1}{(q^2;q^2)_\infty} \sum_{n=0}^\infty q^{n^2+n}\alpha_n^{(2)}(q^2;q^2) \quad  \text{(by \eqref{BP-S2-aq-limit})} \nonumber \\
&=\frac{1}{(q^2;q^2)_\infty}\sum_{k\in\mathbb{Z}}(-1)^k
 \left(q^{24k^2+7k}+q^{24k^2-23k+5}\right) \nonumber \\
&=\frac{1}{(q^2;q^2)_\infty} \mathcal Q(q^5;q^{16}).
\end{align}
Here for the last equality we used \eqref{eq-quintuple}.  This proves \eqref{thm1-id-2}. 

Finally, we have
\begin{align}
F_1(6,4,6)&=(q^2;q^4)_\infty \sum_{n=0}^\infty \frac{q^{n^2+3n}}{(q^2;q^4)_{n+1}} \frac{P_n(q^2)}{(q^2;q^2)_n} \nonumber \\
&=\frac{(q^2;q^4)_\infty}{1-q^2} \sum_{n=0}^\infty q^{n^2+n}(-q^2;q^2)_n \beta_n^{(3)}(q^2;q^2) \nonumber \\
&=\frac{1}{(q^2;q^2)_\infty} \sum_{n=0}^\infty q^{n^2+n}\alpha_n^{(3)}(q^2;q^2)  \quad \text{(by \eqref{BP-S2-aq-limit})}\nonumber \\
&=\frac{1}{(q^2;q^2)_\infty}\sum_{k\in\mathbb{Z}}(-1)^k
 \left(q^{24k^2+13k}+q^{24k^2-29k+7}\right) \nonumber \\
&=\frac{1}{(q^2;q^2)_\infty}\mathcal Q(q^7;q^{16}).
\end{align}
Here for the last equality we used \eqref{eq-quintuple}. This proves \eqref{thm1-id-3}.
\end{proof}

\begin{proof}[Proof of Theorem~\ref{thm-2}]
Set $n=j+2k$. Using \eqref{Euler} to sum over $i$, we deduce that
\begin{align}\label{F2-start}
 F_2(a,b,c)&=\sum_{j,k\geq 0} (-1)^j\frac{q^{j^2+8k^2+4jk+(b-1)j+(c-8)k}}{(q^2;q^2)_j(q^4;q^4)_k}(-q^{2j+4k+a};q)_\infty  \nonumber \\
 &=(-q^a;q)_\infty\sum_{n\geq 0} \frac{(-1)^nq^{n^2+(b-1)n}}{(-q^a;q)_{2n}}
 \sum_{k=0}^{\lfloor n/2\rfloor} \frac{q^{4k^2+(c-2b-6)k}}{(q^2;q^2)_{n-2k}(q^4;q^4)_k}.
\end{align}
Note that the value of $c-2b-6$ is 0 for $(a,b,c)=(1,3,12),(1,1,8)$ and $4$ for $(a,b,c)=(3,3,16)$. Hence the inner sum are connected with $D_n^{(0)}(q^2)$ and $D_n^{(1)}(q^2)$ for these three cases. 

(1) When $(a,b,c)=(1,3,12)$, from \eqref{F2-start} we have 
\begin{align}
    F_2(1,3,12)&=(-q;q)_\infty \sum_{n\geq 0} (-1)^n \frac{q^{n^2+2n}}{(-q;q)_{2n}}  \frac{D_n^{(0)}(q^2)}{(q^2;q^2)_n} \nonumber \\
    &=(-q;q)_\infty\sum_{n\geq 0} (-1)^n q^{n^2}(q;q^2)_n \beta_n^{(5)}(1;q^2).
\end{align}
Substituting the Bailey pair $(\alpha_n^{(5)},\beta_n^{(5)})$ in \eqref{BP5} into \eqref{BP-S2-a1-limit-negative}, we deduce that
\begin{align}
&(q^2;q^2)_\infty F_2(1,3,12)=\sum_{r\geq 0}(-1)^rq^{r^2}\alpha_r^{(5)}(1;q^2) \nonumber \\
&=1+\sum_{m\geq 1}(-1)^m(q^{24m^2-5m}+q^{24m^2+5m})+\sum_{m\geq 0}(-1)^m
 (q^{24m^2+11m+1}-q^{24m^2+37m+14}) \nonumber \\
&=\sum_{k\in\mathbb Z}(-1)^k (q^{24k^2-5k}+q^{24k^2-11k+1})=\mathcal{Q}(q;q^{16}).
\end{align}
This proves \eqref{thm2-id-1}.

(2) When $(a,b,c)=(1,1,8)$, from \eqref{F2-start} we have 
\begin{align}
    F_2(1,1,8)=(-q;q)_\infty \sum_{n\geq 0} (-1)^nq^{n^2}(q;q^2)_n \beta_n^{(4)}(1;q^2).
\end{align}
Substituting the Bailey pair $(\alpha_n^{(4)},\beta_n^{(4)})$ in \eqref{BP4} into \eqref{BP-S2-a1-limit-negative}, we obtain 
\begin{align}
& (q^2;q^2)_\infty F_2(1,1,8)=\sum_{r\geq 0}(-1)^rq^{r^2}\alpha_r^{(4)}(1;q^2) \nonumber \\
&=1+\sum_{m\geq 1}(-1)^m(q^{24m^2-m}+q^{24m^2+m})+\sum_{m\geq 0}(-1)^m
 (q^{24m^2+17m+3}-q^{24m^2+31m+10}) \nonumber \\
 &=\sum_{k\in\mathbb Z}(-1)^k(q^{24k^2+k}+q^{24k^2-17k+3})=\mathcal{Q}(q^3;q^{16}).
\end{align}
This proves \eqref{thm2-id-2}.

(3) When $(a,b,c)=(3,3,16)$, from \eqref{F2-start} we have 
\begin{align}\label{F2-3316-mid}
& F_2(3,3,16)=(-q;q)_\infty\sum_{n\geq 0}
 \frac{(-1)^nq^{n^2}}{(-q;q)_{2n+2}}
 \frac{q^{2n} D_n^{(1)}(q^2)}{(q^2;q^2)_n} \nonumber \\
&=(-q;q)_\infty \frac{1-q}{1-q^2} \sum_{n\geq 0} (-1)^n q^{n^2}(q^3;q^2)_n \beta_n^{(7)}(q^2;q^2) \nonumber \\
&=(-q;q)_\infty \frac{1}{1+q} \frac{(q;q^2)_\infty}{(q^4;q^2)_\infty} \sum_{n=0}^\infty \frac{1-q^{2n+1}}{1-q}(-1)^nq^{n^2}\alpha_n^{(7)}(q^2;q^2).
\end{align}
Here the Bailey pair $(\alpha_n^{(7)},\beta_n^{(7)})$ is given in \eqref{BP7} and we substitute it into \eqref{S4-aq-limit-negative} to get the last equality.
Hence we have
\begin{align}
&(q^2;q^2)_\infty F_2(3,3,16)=\sum_{n=0}^\infty (1-q^{2n+1})(-1)^nq^{n^2}\alpha_n^{(7)}(q^2;q^2) \nonumber \\
&=\sum_{m=0}^\infty (1-q^{6m+1})(-1)^mq^{9m^2}A_m(q^2)   \nonumber \\
&\quad +\sum_{m=0}^\infty (1-q^{6m+3})(-1)^{m}q^{(3m+1)^2}(A_m(q^2)+B_m(q^2)) \nonumber \\
&\quad +\sum_{m=0}^\infty (1-q^{6m+5})(-1)^mq^{(3m+2)^2}B_m(q^2) \nonumber \\
&=\sum_{m=0}^\infty (-1)^mq^{9m^2}(1-q^{12m+4})A_m(q^2)+\sum_{m=0}^\infty (-1)^mq^{(3m+1)^2}(1-q^{12m+8})B_m(q^2) \nonumber \\
&=\sum_{m=0}^\infty (-1)^mq^{24m^2+13m}(1-q^{6m+2})+\sum_{m=0}^\infty (-1)^mq^{24m^2+29m+7}(1-q^{6m+4}).
\end{align}
Replacing $m$ by $-m-1$ in the second sum, we deduce that
\begin{align}\label{F2-3316-Q}
&(q^2;q^2)_\infty F_2(3,3,16)=\sum_{m=-\infty}^\infty (-1)^m q^{24m^2+13m}(1-q^{6m+2}) \nonumber \\
&=\sum_{m=-\infty}^\infty (-1)^m q^{24m^2+13m}-\sum_{m=-\infty}^\infty (-1)^m q^{24m^2+19m+2} \nonumber \\
&=\sum_{m=-\infty}^\infty (-1)^m q^{24m^2+13m}(1+q^{16m+7})=\mathcal{Q}(q^7;q^{16}).
\end{align}
Here we replaced $m$ by $-m-1$ for the second sum to obtain the third equality, and we used  \eqref{eq-quintuple} to obtain the last equality. This proves \eqref{thm2-id-3}.
\end{proof}


\begin{proof}[Proof of Theorem~\ref{thm-3}]
We denote $m=i+2j$ and $n=k+2\ell$. Then the exponent in the  sum $F_3(1,5,1,12)$ can be written as
\begin{align}
 &\binom{m}{2}+m+\binom{i}{2}+2j
 +k^2+2mk+4m\ell+4k\ell+8\ell^2+4\ell \nonumber \\
 &=\binom{m}{2}+m+\binom{i}{2}+2j+2mn+n^2+4\ell^2+4\ell.
\end{align}
Recall the following identity (see \cite[(3.21)]{WW-rank2})
\begin{align}
 c_m:=\sum_{i+2j=m}\frac{q^{\binom{i}{2}+2j}}{(q;q)_i(q^2;q^2)_j}=\frac{1}{(q;q)_m}-\frac{1}{(q;q)_{m-2}}=\frac{q^{m-1}(1+q-q^m)}{(q;q)_m}. \label{eq-cm}
\end{align}
This can be proved easily through the generating function
\begin{align*}
 \sum_{m\geq 0}c_mz^m
 =\frac{(-z;q)_\infty}{(z^2q^2;q^2)_\infty}
 =\frac{1-z^2}{(z;q)_\infty}.
\end{align*}

Applying \eqref{eq-cm} and then using \eqref{Euler} to sum over $m$, we deduce that 
\begin{align}\label{F3-S-F2}
&F_3(1,5,1,12)
 =q^{-1}\sum_{n\geq 0}(-1)^nq^{n^2}
 \big((1+q)(-q^{2n+2};q)_\infty-(-q^{2n+3};q)_\infty\big) \nonumber \\
 &\qquad \qquad \qquad \qquad \times \sum_{k+2\ell =n} \frac{q^{4\ell^2+4\ell}}{(q^2;q^2)_k(q^4;q^4)_\ell}\nonumber \\
 &=(-q;q)_\infty\sum_{n\geq 0}
 \frac{(-1)^nq^{n^2}\big(1+(1+q)q^{2n+1}\big)}{(-q;q)_{2n+2}} \times \frac{D_n^{(1)}(q^2)}{(q^2;q^2)_{n}} \nonumber \\
 &=(-q;q)_\infty\sum_{n\geq 0}
 \frac{(-1)^nq^{n^2}(q;q^2)_{n+1}}{(q^2;q^2)_{2n+1}}\Big(1+\frac{q^{2n+1}}{1+q^{2n+2}}\Big)D_n^{(1)}(q^2) \nonumber \\
 &=\frac{1-q}{1-q^2}(-q;q)_\infty \sum_{n=0}^\infty (-1)^n q^{n^2}(q^3;q^2)_{n}\Big(\beta_n^{(6)}(q^2;q^2)+q\beta_n^{(7)}(q^2;q^2)\Big) \nonumber \\
 &=S(q)+qF_2(3,3,16).
\end{align}
Here we recognize the second sum involving $\beta_n^{(7)}(q^2;q^2)$ as $F_2(3,3,16)$ (see \eqref{F2-3316-mid}) and denote the first sum by
\begin{align}
&S(q):=\frac{1-q}{1-q^2}(-q;q)_\infty \sum_{n=0}^\infty (-1)^n q^{n^2}(q^3;q^2)_{n}\beta_n^{(6)}(q^2;q^2).
\end{align}
Substituting \eqref{BP6} into \eqref{S4-aq-limit-negative}, we deduce that
\begin{align}
&(q^2;q^2)_\infty S(q)=\sum_{n=0}^\infty (-1)^nq^{n^2}(1-q^{2n+1})\alpha_n^{(6)}(q^2;q^2) \nonumber \\
&=\sum_{m=-\infty}^\infty (-1)^mq^{24m^2+7m}(1-q^{6m+1})+\sum_{m=-\infty}^\infty (-1)^mq^{24m^2+19m+3}(1-q^{6m+3}).
\end{align}
Replacing $m$ by $-m-1$ in the second sum, we deduce that
\begin{align}\label{S-result}
&(q^2;q^2)_\infty S(q)=\sum_{m=-\infty}^\infty (-1)^mq^{24m^2+7m}(1+q^{16m+5})-q\sum_{m=-\infty}^\infty q^{24m^2+13m}(1+q^{16m+7}) \nonumber \\
&=\mathcal{Q}(q^5;q^{16})-q\mathcal{Q}(q^7;q^{16}).
\end{align}
Substituting \eqref{S-result} and \eqref{F2-3316-Q} into \eqref{F3-S-F2}, we obtain \eqref{thm3-id}.
\end{proof}

\section{Proof of the multisum identity}\label{sec-thm-general}
Before giving a proof for Theorem \ref{thm-general}, we make some preparations. 
Throughout this section we let $\omega$ be a primitive cube root of unity. For every $m\in\mathbb{Z}$, we define
\begin{equation}\label{eta-A-defn}
\mathcal{A}_m= \mathcal{A}_m(q)=\eta_m q^{m(m+1)/2}
\end{equation}
where
\begin{align}\label{eta-value}
    \eta_m:=\frac{(-1)^m(\omega^{m+1}-\omega^{2m+2})}{\omega-\omega^2}=\begin{cases}
 (-1)^r, & m=3r, \\
 (-1)^r, & m=3r+1, \\
 0,  & m=3r+2.  
  \end{cases}
\end{align}
For $M\geq 0$, we define
\begin{equation}\label{eq-B-definition}
 \mathcal{B}_M=\mathcal{B}_M(q):=\sum_{\substack{i,j,k,h\geq 0\\i+3j+3k+2h=M}}
 \frac{q^{ih+i+j+2k}}
 {(q;q)_i(q^3;q^3)_j(q^3;q^3)_k(q;q)_h}.
\end{equation}

\begin{lemma}\label{lem-B-product}
For $|x|<1$ we have
\begin{equation}\label{eq-B-product}
 \sum_{M\geq 0} \mathcal{B}_M(q)x^M
 =\frac{(x^3q^3;q^3)_\infty}{(xq;q)_\infty(x^2;q)_\infty}
 =\frac{(\omega x q,\omega^2xq;q)_\infty}{(x^2;q)_\infty}.
\end{equation}
\end{lemma}
\begin{proof}
Summing over $j$ and $k$ first and then over $i$ using \eqref{Euler} in \eqref{eq-B-definition}, we have
\begin{align*}
 \sum_{M\geq 0} \mathcal{B}_M(q)x^M
&=\frac{1}{(x^3q,x^3q^2;q^3)_\infty}
 \sum_{i,h\geq 0}\frac{q^{ih+i}x^{i+2h}}{(q;q)_i(q;q)_h} \nonumber \\
 &=\frac{1}{(x^3q,x^3q^2;q^3)_\infty(xq;q)_\infty}
   \sum_{h\geq 0}\frac{(xq;q)_h x^{2h}}{(q;q)_h}\\
 &=\frac{(x^3q;q)_\infty}
 {(x^3q,x^3q^2;q^3)_\infty(xq;q)_\infty(x^2;q)_\infty} \quad \text{(by \eqref{q-binomial})} \\
 &=\frac{(x^3q^3;q^3)_\infty}{(xq;q)_\infty(x^2;q)_\infty}. \qedhere
\end{align*}
\end{proof}

\begin{lemma}\label{lem-finite-theta}
For every $M\geq 0$,
\begin{equation}\label{eq-finite-theta}
 \mathcal{B}_M(q)=\sum_{\substack{-M\le m\le M\\m\equiv M \!\! \pmod2}}
 \frac{\mathcal{A}_m(q)}{(q;q)_{(M-m)/2}(q;q)_{(M+m)/2}}.
\end{equation}
\end{lemma}
\begin{proof}
For $a\ne0$ and $|x|<1$, we need the following known identity \cite[Theorem 3.1]{Cao-Wang2023}:
\begin{align}\label{G-id}
 G(a,x):=\sum_{i,j\geq 0}
 \frac{(-a)^{i-j}q^{(i-j)(i-j-1)/2}x^{i+j}}{(q;q)_i(q;q)_j} =\frac{(ax,xq/a;q)_\infty}{(x^2;q)_\infty}.
\end{align}
This can be proved directly (different from the proof in \cite{Cao-Wang2023}) by using \eqref{Euler} and \eqref{q-binomial}: 
\begin{align}
 G(a,x)
 =\sum_{i\geq 0}\frac{(-ax)^i q^{i(i-1)/2}}{(q;q)_i}
      (xq^{1-i}/a;q)_\infty=(xq/a;q)_\infty\sum_{i\geq 0}\frac{(a/x;q)_i x^{2i}}{(q;q)_i}.
 \label{eq-Euler-double}
\end{align}
Applying \eqref{q-binomial} to the last sum we obtain \eqref{G-id}. 

From \eqref{eta-A-defn} and \eqref{G-id} we deduce that
\begin{align*}
 &\sum_{i,j\geq 0}\frac{\mathcal{A}_{i-j}x^{i+j}}{(q;q)_i(q;q)_j}
 =\frac{\omega G(\omega q,x)-\omega^2G(\omega^2q,x)}
 {\omega-\omega^2} \nonumber \\
&= \frac{\omega (\omega xq,\omega^2x;q)_\infty-\omega^2(\omega^2xq ,\omega x;q)_\infty}{(\omega-\omega^2)(x^2;q)_\infty} \nonumber \\
 &=\frac{(\omega xq,\omega^2xq;q)_\infty \big(\omega(1-\omega^2 x)-\omega^2 (1-\omega x)\big)}{(\omega-\omega^2)(x^2;q)_\infty} \nonumber \\
 &=\frac{(\omega xq,\omega^2xq;q)_\infty}{(x^2;q)_\infty}.
\end{align*}
Comparing this with \eqref{eq-B-product}, using $M=i+j$ and $m=i-j$, we obtain \eqref{eq-finite-theta}.
\end{proof}

\begin{lemma}\label{lem-general-BP}
For $\sigma=0,1$, the sequences
\begin{equation}\label{eq-general-BP}
 \beta_r^{(\sigma)}=\mathcal{B}_{2r+\sigma}(q),\quad
 \alpha_r^{(\sigma)}=
 \frac{1-q^{2r+\sigma+1}}{(q;q)_{\sigma+1}} \mathcal{A}_{2r+\sigma}(q)
 \quad(r\geq 0)
\end{equation}
form a Bailey pair relative to $a=q^{\sigma+1}$. In other words,
\begin{equation}\label{eq-Bailey-relation}
 \mathcal{B}_{2N+\sigma}(q)=\sum_{r=0}^N
 \frac{\alpha_r^{(\sigma)}}{(q;q)_{N-r}(q^{\sigma+2};q)_{N+r}}
 \quad (N\geq 0).
\end{equation}
\end{lemma}
\begin{proof}
The sequence in \eqref{eta-A-defn} satisfies
\begin{equation}\label{eq-A-reflection}
 \mathcal{A}_{-m-2}(q)=-q^{m+1}\mathcal{A}_m(q)\quad(m\in\mathbb{Z}),\quad \mathcal{A}_{-1}(q)=0.
\end{equation}
We first rewrite \eqref{eq-finite-theta} as
\begin{equation}\label{eq-finite-theta-rewrite}
 \mathcal{B}_M(q)=\sum_{\substack{-M-2\le m\le M\\m\equiv M \!\! \pmod2}}
 \frac{\mathcal{A}_m(q)(1-q^{(M+m)/2+1})}{(q;q)_{(M-m)/2}(q;q)_{(M+m)/2+1}}.
\end{equation}
Here the term corresponding to $m=-M-2$ vanishes.
The index set of $m$ is now invariant under $m\mapsto-m-2$.
Pairing these indices, the two factorials in the denominator are interchanged, while by \eqref{eq-A-reflection} the combined numerator is
\begin{align*}
\mathcal{A}_m(q)\big(1-q^{(M+m)/2+1}\big)
   +\mathcal{A}_{-m-2}(q)\big(1-q^{(M-m)/2}\big)=(1-q^{m+1})\mathcal{A}_m(q).
\end{align*}
In particular, when $m=-1$, the numerator vanishes. Thus
\begin{equation}\label{eq-BM-simplify}
 \mathcal{B}_M(q)=\sum_{\substack{0\le m\le M\\m\equiv M \!\! \pmod2}}
 \frac{(1-q^{m+1})\mathcal{A}_m(q)}
 {(q;q)_{(M-m)/2}(q;q)_{(M+m)/2+1}}.
\end{equation}
Substituting $M=2N+\sigma$ and $m=2r+\sigma$, and using the fact
\begin{align*}
 (q;q)_{N+r+\sigma+1}=(q;q)_{\sigma+1}(q^{\sigma+2};q)_{N+r},
\end{align*}
we obtain \eqref{eq-Bailey-relation}.
\end{proof}

\begin{proof}[Proof of Theorem~\ref{thm-general}]
We denote by $W_{n,b}(q)$ the left-hand side of \eqref{eq-thm-general}. For each summand, set
\begin{equation}\label{eq-change}
 d=i+3j+3k,\quad K_r=L_r+\frac{d}{2} \quad (1\leq r\leq 2n).
\end{equation}
Completing squares in \eqref{eq-quadratic} and \eqref{eq-linear} gives
\begin{equation}\label{eq-complete-square}
 \frac{Q_n+R_{n,b}}2
 =\sum_{r=1}^{2n} K_r^2+\sum_{r=2n-b+1}^{2n} K_r
   +iL_{2n}+i+j+2k.
\end{equation}

Moreover,
\begin{align*}
 K_r-K_{r+1}=\ell_r\quad(1\le r<2n),\quad
 2K_{2n}=i+3j+3k+2L_{2n}.
\end{align*}
For fixed $K_{2n}$, the sum over $i,j,k,L_{2n}$ is therefore exactly $\mathcal{B}_{2K_{2n}}(q^2)$ from \eqref{eq-B-definition}. Conversely, given $K_1\ge\cdots\ge K_{2n}$ of a common integer or half-integer parity and a term of $B_{2K_{2n}}$, set $\ell_r=K_r-K_{r+1}$ for $r<2n$ and $\ell_{2n}=h$. Then $L_r=K_r-d/2$ for every $r$.

Separating the two parities and writing $K_r=N_r+\sigma/2$, we obtain
\begin{align}
 W_{n,b}(q)&=\sum_{\sigma=0}^1 q^{(n+b)\sigma}\sum_{N_1\ge\cdots\ge N_{2n}\geq 0} \frac{q^{2(\sum_{r=1}^{2n}N_r^2+\sigma\sum_{r=1}^{2n} N_r +\sum_{r=2n-b+1}^{2n} N_r)}
       \mathcal{B}_{2N_{2n}+\sigma}(q^2)}{\prod_{r=1}^{2n-1}(q^2;q^2)_{N_r-N_{r+1}}}.
 \label{eq-W-mid}
\end{align}
We now apply Lemma~\ref{lem-lattice} with $q$ replaced by $q^2$, $a=q^{2(\sigma+1)}$, $k=2n$ and $j=2n-b$ to the  Bailey pair \eqref{eq-general-BP}. Its left-hand side is the inner sum in \eqref{eq-W-mid}.
On the right, the denominator $1-aq^{4r}$ cancels against the
corresponding factor in $\alpha_r^{(\sigma)}$, and
\begin{align*}
 (q^{2(\sigma+2)};q^2)_\infty(q^2;q^2)_{\sigma+1}=(q^2;q^2)_\infty.
\end{align*}
Hence we have
\begin{align*}
W_{n,b}(q)=\frac{1}{(q^2;q^2)_\infty}\sum_{\sigma=0}^1 q^{(n+b)\sigma}\sum_{r\geq 0} q^{4nr^2+4n\sigma+2br} (1-q^{(4r+2\sigma+2)(2n-b+1)}) \mathcal{A}_{2r+\sigma}(q^2).
\end{align*}
Combining the two parities by setting $m=2r+\sigma$, we obtain
\begin{equation}\label{eq-unilateral}
 W_{n,b}(q)=\frac{1}{(q^2;q^2)_\infty}\sum_{m\geq 0} \mathcal{A}_m(q^2) q^{nm^2+bm}
 \big(1-q^{2(m+1)(2n-b+1)}\big).
\end{equation}
From \eqref{eq-A-reflection} we have
\begin{align*}
 \mathcal{A}_{-m-2}(q^2)q^{n(-m-2)^2+b(-m-2)}
 &=-\mathcal{A}_m(q^2) q^{nm^2+bm+2(m+1)(2n-b+1)}.
\end{align*}
Since $\mathcal{A}_{-1}=0$, after replacing $m$ by $-m-2$ in the second term in \eqref{eq-unilateral}, we obtain
\begin{equation}\label{eq-bilateral}
 W_{n,b}(q)=\frac{1}{(q^2;q^2)_\infty}\sum_{m\in\mathbb{Z}}\mathcal{A}_m(q^2) q^{nm^2+bm}.
\end{equation}
Substituting \eqref{eta-A-defn} and \eqref{eta-value} into it, we have
\begin{align}
&(q^2;q^2)_\infty W_{n,b}(q)=\sum_{m\in\mathbb{Z}}\eta_m q^{(n+1)m^2+(b+1)m} \nonumber \\
 &=\sum_{r\in\mathbb{Z}}(-1)^r\Big(q^{9(n+1)r^2+3(b+1)r} +q^{(n+1)(1-3r)^2+(b+1)(1-3r)}\Big) \nonumber \\
&=\sum_{r\in\mathbb{Z}}(-1)^r q^{9(n+1)r^2-3(n+1)r}\Big(q^{3(n+b+2)r}+q^{(n+b+2)(1-3r)}  \Big) \nonumber \\
 &=\mathcal{Q}(q^{n+b+2};q^{6n+6}). \label{eq-theta-final}
\end{align}
Here for the last equality we used \eqref{eq-quintuple}. This proves \eqref{eq-thm-general}.
\end{proof}

\subsection*{Acknowledgements}
This work was supported by the National Key R\&D Program of China (Grant No.\ 2024YFA1014500).

\subsection*{Declaration of AI use}
GPT-5.6 Sol and GPT-6 Astra were used to assist in exploring proof strategies, performing calculations, and polishing the language. All AI-assisted arguments and suggestions were independently verified and revised as necessary by the author, who takes full responsibility for the content of the article.

\end{document}